\documentclass{amsart}

\usepackage{amsmath,amsfonts,amssymb,amsthm,graphicx,cite}
\usepackage{amscd}
\usepackage{mathrsfs}
\usepackage{lineno}
\usepackage{float}
\usepackage{caption}
\usepackage{subcaption}

\newtheorem{theorem}{Theorem}[section]

\newtheorem{corollary}[theorem]{Corollary}
\newtheorem{lemma}{Lemma}[section]
\newtheorem{proposition}{Proposition}[section]
\theoremstyle{remark}
\newtheorem{remark}{Remark}[section]
\newtheorem{definition}{Definition}[section]

\newcommand{\De}{\Delta}

\newcommand{\Z}{{\mathbb Z}}
\newcommand{\N}{{\mathbb N}}
\newcommand{\C}{{\mathbb C}}
\newcommand{\R}{{\mathbb R}}

\newcommand{\cH}{{\mathcal H}}
\newcommand{\cL}{{\mathcal L}}

\newcommand{\cZ}{{\mathcal Z}}
\newcommand{\measure}{\omega}

\DeclareMathOperator{\spn}{span}

\begin{document}

\title[Orthogonality of super-Jack polynomials]{Orthogonality of super-Jack polynomials and a Hilbert space interpretation of deformed Calogero-Moser-Sutherland operators}

\author{Farrokh Atai}
\address{Department of Mathematics, Kobe University, Rokko, Kobe 657-8501, Japan}
\email{farrokh@math.kobe-u.ac.jp}

\author{Martin Halln\"as}
\address{Department of Mathematical Sciences, Chalmers University of Technology and the University of Gothenburg, SE-412 96 Gothenburg, Sweden}
\email{hallnas@chalmers.se}

\author{Edwin Langmann}
\address{Department of Physics, KTH Royal Institute of Technology, SE-106 91 Stockholm, Sweden}
\email{langmann@kth.se}

\date{\today}

\begin{abstract}
We prove orthogonality and compute explicitly the (quadratic) norms for super-Jack polynomials $SP_\lambda((z_1,\ldots,z_n),(w_1,\ldots,w_m);\theta)$ with respect to a natural positive semi-definite, but degenerate, Hermitian product $\langle\cdot,\cdot\rangle_{n,m}^\prime$. In case $m=0$ (or $n=0$), our product reduces to Macdonald's well-known inner product $\langle\cdot,\cdot\rangle_n^\prime$, and we recover his corresponding orthogonality results for the Jack polynomials $P_\lambda((z_1,\ldots,z_n);\theta)$. From our main results, we readily infer that the kernel of $\langle\cdot,\cdot\rangle_{n,m}^\prime$ is spanned by the super-Jack polynomials indexed by a partition $\lambda$ not containing the $m\times n$ rectangle $(m^n)$. As an application, we provide a Hilbert space interpretation of the deformed trigonometric Calogero-Moser-Sutherland operators of type $A(n-1,m-1)$.
\end{abstract}

\maketitle

\section{Introduction}\label{sec1}
The celebrated Jack polynomials $P_\lambda((z_1,\ldots,z_n);\theta)$, depending on a partition $\lambda$ and a parameter $\theta$, were introduced by Henry Jack \cite{Jac70} to interpolate between Schur polynomials ($\theta=1$) and zonal polynomials ($\theta=1/2$). (Throughout the paper, we use the inverse $\theta=1/\alpha$ of the standard parameter $\alpha$ for the Jack polynomials.) Many of the properties of Schur- and zonal functions were subsequently generalised to the Jack case by Macdonald \cite{Mac87,Mac95}, Stanley \cite{Sta89} and others; and important applications of the Jack polynomials can, by now, be found in areas such as (quantum) integrable systems of Calogero-Moser-Sutherland (CMS) type, $\beta$-ensembles of random matrix theory, Hilbert schemes of points in algebraic geometry, the Alday-Gaiotto-Tachikawa (AGT) correspondence in quantum field theory, and many more; see e.g.~\cite{vDV00,For10,Nak16,MS14} and references therein. 

In this paper, we focus on a particular generalisation of the Jack polynomials, namely the super-Jack polynomials $SP_\lambda((z_1,\ldots,z_n),(w_1,\ldots,w_m);\theta)$, which were introduced by Kerov, Okounkov and Olshanski \cite{KOO98} in the context of discrete potential theory on the so-called Young graph. Since then, the super-Jack polynomials have appeared in constructions of (formal) eigenfunctions of deformed CMS operators \cite{Ser00,Ser02,SV05,HL10,DH12}, in studies of circular $\beta$-ensembles of random matrices \cite{Oko97,Mat08,DL15}, as well as a conformal field theory description of the fractional quantum Hall effect \cite{AL17}.

The algebraic- and combinatorial properties of super-Jack polynomials have been studied in great detail, but, in stark contrast to the Jack polynomials, an interpretation as orthogonal polynomials has been missing. Our main result is a natural generalisation of Macdonald's well-known inner product $\langle\cdot,\cdot\rangle_n^\prime$, and corresponding orthogonality relations and (quadratic) norms formulae for the Jack polynomials, to the super case. As a particular application of these results, we are able to provide a Hilbert space interpretation of deformed trigonometric CMS operators of type $A(n-1,m-1)$, something that, due to singularities of the relevant eigenfunctions, up until now has proved to be elusive.

From work of Macdonald \cite{Mac87,Mac95}, it is well-known that the Jack polynomials $P_\lambda(z;\theta)$, $z= (z_1,\ldots,z_n)$, form an orthogonal system on the $n$-dimensional torus $T^n= T^n_1$, where
\begin{equation}\label{tori}
T_\xi^n = \{z\in\C^n : |z_j|=\xi,\, j=1,\ldots,n\},\ \ \ \xi>0,
\end{equation}
with respect to the weight function
\begin{equation}\label{Deln}
\Delta_n(z;\theta) = \prod_{1\leq j<k\leq n}\big[(1-z_j/z_k)(1-z_k/z_j)\big]^\theta, 
\end{equation}
and remarkably simply explicit formulae for the corresponding norms of the Jack polynomials are known; see \eqref{Porth}--\eqref{Nlam}. In this paper, we prove that the super-Jack polynomials $SP_\lambda(z,w;\theta)$, $z=(z_1,\ldots,z_n)$ and $w=(w_1,\ldots,w_m)$,  are pair-wise orthogonal with respect to a sesquilinear product with weight function
\begin{equation}\label{Delnm}
\Delta_{n,m}(z,w;\theta) = \frac{\De_n( z;\theta)\De_m(w;1/\theta)}{\prod_{j=1}^n\prod_{k=1}^m(1-z_j/w_k)(1-w_k/z_j)}, 
\end{equation}
and we explicitly compute the corresponding (quadratic) norms; see Theorem \ref{Thm:orth}. To avoid the poles along $z_j=w_k$, we integrate the variables $z$ and $w$ over tori $T_\xi^n$ resp.~$T_{\xi^\prime}^m$ with $\xi\neq \xi^\prime$. Due (only) to the denominator in \eqref{Delnm}, we are thus working with a complex-valued weight function (``only'' since $\Delta_n(z;\theta)\geq 0$ on $T_\xi^n$). Nevertheless, we show that the norms in question are always non-negative. To be more specific, it is known that  the super-Jack polynomials vanish identically unless the partition $\lambda$ satisfies the fat-hook condition \cite{SV05}
\begin{equation}\label{fathook}
\lambda_{n+1}\leq m,
\end{equation}
and while the norms we obtain are non-negative for all these partitions, we find that they are positive if and only if
\begin{equation}\label{box}
\lambda_{n+1}\leq m\leq \lambda_n . 
\end{equation} 
Note that this condition is equivalent to adding $(m^n)\subseteq\lambda$ to the former condition \eqref{fathook}. As discussed below, the requirement \eqref{box} has a natural interpretation in a recently proposed physics application of the super-Jack polynomials \cite{AL17}. 

Sergeev and Veselov \cite{SV05} proved that the super-Jack polynomials $SP_\lambda( z,w ;\theta)$ form a linear basis in the subalgebra
\begin{equation}\label{Lam}
\Lambda_{n,m,\theta}\subset\C[z,w]^{S_n\times S_m}
\end{equation}
consisting of all polynomials $p(z,w)$ that are $S_n$-invariant in $z$, $S_m$-invariant in $w$,  and furthermore satisfy the quasi-invariance condition
\begin{equation}\label{qinv}
\left. \left(\frac{\partial}{\partial z_j} + \theta\frac{\partial}{\partial w_k}\right)p(z,w)\right|_{z_j=w_k}\equiv 0
\end{equation}
for all $j=1,\ldots,n$ and $k=1,\ldots,m$. 
From our orthogonality results and norms formulae, it readily follows that these invariance properties are sufficient to ensure that the sesquilinear product in question is both Hermitian and independent of the choice of $\xi, \xi^\prime>0$ as long as $\xi\neq\xi^\prime$. Moreover, as indicated above, we find that the product is positive semidefinite, and that its kernel is  spanned by the super-Jack polynomials $SP_\lambda(z,w;\theta)$ for which $\lambda$ does not contain the $m\times n$ rectangle $(m^n)$. Consequently, it descends to a positive definite inner product on the factor space
\begin{equation*} 
V_{n,m,\theta}:= \Lambda_{n,m,\theta}\big/\spn\{ SP_\lambda(z,w;\theta) : (m^n)\not\subseteq\lambda \}.
\end{equation*} 
(Here and below, we write $\spn S$ for the $\C$-linear span of a set $S$.) 

Our main motivation for studying the orthogonality properties of super-Jack polynomials is due to their close connection with the deformed CMS operator
\begin{equation}
\label{dCMSOp}
\begin{split}
\cH_{n,m,\theta} &= -\sum_{j=1}^n\frac{\partial^2}{\partial x_j^2} + \sum_{1\leq j<k\leq n}\frac{\theta(\theta-1)}{2\sin^2\frac{1}{2}(x_j-x_k)}\\
&+ \theta\sum_{k=1}^m\frac{\partial^2}{\partial y_k^2} - \sum_{1\leq j<k\leq m}\frac{(\theta^{-1}-1)}{2\sin^2\frac{1}{2}(y_j-y_k)} + \sum_{j=1}^n\sum_{k=1}^m\frac{(1-\theta)}{2\sin^2\frac{1}{2}(x_j-y_k)}.
\end{split}
\end{equation}
The $m=1$ instance of this operator was introduced, and shown to be integrable, in 1998 by Chalykh, Feigin and Veselov \cite{CFV98}. Shortly thereafter, Sergeev \cite{Ser00,Ser02} wrote down the operator for all $m>1$ and related it to the root system of the Lie superalgebra $\mathfrak{gl}(n|m)$. Furthermore, he showed that $\cH_{n,m,\theta}$ has (formal) eigenfunctions of the form
\begin{equation}
\label{Psi}
\Psi_\lambda(x,y;\theta) = \Psi_0(x,y;\theta)SP_\lambda\big(z(x),w(y);\theta\big)
\end{equation}
with
$$
z(x) = (e^{ix_1},\ldots,e^{ix_n}),\ \ \ w(y) = (e^{iy_1},\ldots,e^{iy_m})
$$
and
$$
\Psi_0(x,y;\theta) = \frac{\left(\prod_{1\leq j<k\leq n}\sin^2\frac{1}{2}(x_j-x_k)\right)^{\theta/2}\left(\prod_{1\leq j<k\leq m}\sin^2\frac{1}{2}(y_j-y_k)\right)^{1/2\theta}}{\prod_{j=1}^n\prod_{k=1}^m\sin\frac{1}{2}(x_j-y_k)}.
$$
Sergeev and Veselov proved in \cite{SV04} the integrability of $\cH_{n,m,\theta}$ for $m>1$ by an explicit recursive construction of higher order quantum integrals, and in their subsequent paper \cite{SV05}, which contains a more conceptual proof of integrability, they demonstrated that the functions $\Psi_\lambda$ are joint eigenfunctions of all quantum integrals.

From the point of view of (functional) analysis and quantum physics, these results were difficult to interpret: Regardless of the value of $\theta$, the eigenfunctions $\Psi_\lambda(x,y;\theta)$ are not contained in the Hilbert space $L^2([0,2\pi]^n\times [0,2\pi]^m)$ unless $n=0$ or $m=0$, and when attempting to interpret $\cH_{n,m,\theta}$ as the Schr\"odinger operator of a quantum many-body system one encounters the problem that $m$ particles have negative mass as long as $\theta>0$.

Our results provide a solution to the former problem: From our inner product on $V_{n,m,\theta}$, we obtain a natural regularisation of the ill-defined $L^2$-products of the eigenfunctions $\Psi_\lambda$, which is independent of the specific choice of regularisation parameters. As a consequence, we are able to promote the deformed CMS operator \eqref{dCMSOp}, as well as each of the higher order quantum integrals, to a self-adjoint operator in a Hilbert space, corresponding to the completion of $V_{n,m,\theta}$, with an orthogonal basis given by the eigenfunctions $\Psi_\lambda$ with $\lambda$ satisfying \eqref{box}.

A possible resolution of the latter problem, i.e.~that of negative masses, was recently proposed by two of us in \cite{AL17}. As shown there, the deformed CMS operator \eqref{dCMSOp} provides an effective description of a conformal field theory (CFT) relevant for the fractional quantum Hall effect in a subspace corresponding to $n$ particles and $m$ holes. States in this subspace can be characterised by the particle- and hole momenta, which are given by partitions $\mu=(\mu_1,\ldots,\mu_n)$ and $\nu=(\nu_1,\ldots,\nu_m)$ of lengths $\leq n$ and $\leq m$, respectively. Moreover, such a state can be naturally associated with the eigenfunction $\Psi_\lambda(x,y;\theta)$ \eqref{Psi} with
\begin{equation} 
\label{lambdamunu} 
\lambda=(\mu+(m^n),\nu'). 
\end{equation}

Note that the partitions of this form are precisely those satisfying the condition \eqref{box}. In \cite{AL17} super-Jack polynomials were used to construct CFT states, and the polynomials corresponding to a partition $\lambda$ satisfying \eqref{fathook} but not \eqref{box} turned out to be superfluous. However, this redundancy of states was somewhat puzzling; see \cite{AL17}, end of Section IV.B. Our result that the norms of the super-Jack polynomials in question are zero, in combination with results from \cite{AL17}, entails that these superfluous states are zero, which nicely resolves this puzzle.

Our plan for the rest of this paper is as follows. In Section~\ref{sec2}, we review definitions and results relating to Jack- and super-Jack polynomials that we make use of in the paper. Section~\ref{sec3} is the heart of the paper, in which we formulate and prove the orthogonality relations and norms formulae for the super-Jack polynomials. As an application of our main results, we formulate a Hilbert space interpretation of the deformed CMS operator \eqref{dCMSOp} in Section \ref{sec4}. Section~\ref{sec5} contains final remarks and an outlook.

\section{Jack and super-Jack polynomials}\label{sec2}
In this section, we briefly review features of Jack and super-Jack polynomials we make use of. We follow to a large extent the notation in Macdonald's book \cite{Mac95}, with a notable exception being our use of the inverse
$$
\theta=1/\alpha
$$
of Macdonald's parameter $\alpha$.

\subsection{Symmetric functions and Jack polynomials}
We start with the set of partitions $\mathscr{P}$, consisting of the (finite or infinite) sequences
$$
\lambda = (\lambda_1,\lambda_2,\ldots,\lambda_j,\ldots),\ \ \ \lambda_j\in\Z_{\geq 0},
$$
with
$$
\lambda_1\geq\lambda_2\geq\cdots\geq\lambda_j\geq\cdots
$$
and only finitely many $\lambda_j$ non-zero. (We shall not distinguish between partitions differing only by a string of zeros.) The number of non-zero terms (or parts) $\lambda_j$, denoted $\ell(\lambda)$, is called the length of $\lambda$, and if $|\lambda|:=\lambda_1+\lambda_2+\cdots=k$ one says that $\lambda$ is a partition of $k$. The set consisting of all partitions of $k\in\Z_{\geq 0}$ is denoted $\mathscr{P}_k$. Considering the diagram of a partition $\lambda$, reflection in the diagonal yields the conjugate $\lambda^\prime$ of $\lambda$.

For our purposes, it will be important to work over the complex numbers. Therefore, we let
$$
\Lambda_n = \C[z_1,\ldots,z_n]^{S_n}
$$
denote the algebra of complex symmetric polynomials in $n$ independent variables. Its (homogeneous) bases are naturally labelled by the partitions $\lambda=(\lambda_1,\lambda_2,\ldots,\lambda_n)$ of length $\leq n$. One of the many important such bases consists of the monomial symmetric polynomials
$$
m_\lambda(z_1,\ldots,z_n) = \sum_{a\in S_n(\lambda)}z_1^{a_1}\cdots z_n^{a_n},
$$
where $S_n(\lambda)$ denotes the $S_n$-orbit of $\lambda$. Note that $m_\lambda$ is homogeneous of degree $|\lambda|$. For any $k\in \Z_{\geq 0}$, the set $\{m_\lambda\}_{|\lambda|=k}$, which is a linear basis in the homogeneous component of $\Lambda_n$ of degree $k$, inherits a partial order from the dominance partial order on $\mathscr{P}_k$, defined by
\begin{equation}\label{domOrd}
\mu\leq\lambda\Leftrightarrow \mu_1+\cdots+\mu_j\leq \lambda_1+\cdots+\lambda_j,\, 
\quad \forall j\geq 1.
\end{equation}

We turn now to the symmetric polynomials introduced by Jack \cite{Jac70}. To begin with, we describe one of their definitions due to Macdonald \cite{Mac87,Mac95}, which is closely related to our main result. Recalling the weight function $\Delta_n$ \eqref{Deln}, we consider the inner product $\langle\cdot,\cdot\rangle^\prime_{n,\theta}$ on $\Lambda_n$ defined by
\begin{equation}\label{MacProd}
\langle p,q\rangle^\prime_{n,\theta} = \frac{1}{n!}\int_{T^n}p(z)\overline{q(z)}\Delta_n(z;\theta)d\measure_n(z),
\end{equation}
where $T^n=T^n_1$ (c.f.~\eqref{tori}) and
$$
{d\measure_n(z)} := \frac{1}{(2\pi i)^n}\frac{dz_1}{z_1}\cdots \frac{dz_n}{z_n}.
$$
Given a partition $\lambda=(\lambda_1,\ldots,\lambda_n)$, the (monic) Jack polynomial $P_\lambda(z;\theta)$ can be uniquely characterised by the following two properties:
\begin{enumerate}
\item $P_\lambda=m_\lambda+\sum_{\mu<\lambda}u_{\lambda\mu}m_\mu$,
\item $\langle P_\lambda,m_\mu\rangle^\prime_{n,\theta}=0$ for all $\mu<\lambda$.
\end{enumerate}
From these properties, it is clear that $\langle P_\lambda,P_\mu\rangle^\prime_{n,\theta}=0$ whenever $\mu<\lambda$ or $\mu>\lambda$. Since the dominance order \eqref{domOrd} is only a partial order, it is a remarkable fact that pair-wise orthogonality holds true not only in these cases but for all $\mu\neq\lambda$, namely
\begin{equation}\label{Porth}
\langle P_\lambda,P_\mu\rangle^\prime_{n,\theta} = \delta_{\lambda\mu}N_n(\lambda;\theta)
\end{equation}
with (quadratic) norms
\begin{equation}\label{Nlam}
N_n(\lambda;\theta) = \prod_{1\leq j<k\leq n}\frac{\Gamma(\lambda_j-\lambda_k+\theta(k-j+1))\Gamma(\lambda_j-\lambda_k+\theta(k-j-1)+1)}{\Gamma(\lambda_j-\lambda_k+\theta(k-j))\Gamma(\lambda_j-\lambda_k+\theta(k-j)+1)},
\end{equation}
(where the Kronecker delta $\delta_{\lambda\mu}$ equals $1$ if $\lambda=\mu$ and $0$ otherwise,  and $\Gamma$ is Euler's Gamma function). Before proceeding further, we detail an alternative expression for $\langle p,q\rangle^\prime_{n,\theta}$, which we shall make use of in Section \ref{sec3}. We observe that \eqref{MacProd} can be rewritten as
\begin{equation}\label{MacProdAlt}
\langle p,q\rangle^\prime_{n,\theta} = \frac{1}{n!}\int_{T^n_\xi}p(z)\overline{q(\bar{z}^{-1})}\Delta_n(z;\theta)d\measure_n(z),\ \ \ \forall \xi>0.
\end{equation}
To see this, note that
\begin{equation}\label{scaleInv}
\Delta_n(\xi z) = \Delta_n(z),\ \ \ d\measure_n(\xi z) = d\measure_n(z),\ \ \ \xi>0,
\end{equation}
$$
\overline{q(z)} = \overline{q(\bar{z}^{-1})},\ \ \ z\in T^n
$$
where $\bar{z}^{-1}:=(1/\bar{z}_1,\ldots,1/\bar{z}_n)$, and that homogeneous components of $\Lambda_n$ of different degrees are orthogonal.

A prominent feature of the Jack polynomials is their stability:
$$
P_\lambda(z_1,\ldots,z_{n-1},0) = P_\lambda(z_1,\ldots,z_{n-1}),
$$
where $P_\lambda(z_1,\ldots,z_{n-1})\equiv 0$ if $\ell(\lambda)=n$. Hence, for each partition $\lambda$, there is a well-defined element $P_\lambda(\theta)$ of the algebra
$$
\Lambda = \varprojlim_n\, \Lambda_n
$$
of complex symmetric functions. (The inverse limit should be taken in the category of graded algebras relative to the restriction homomorphisms $\rho_n:\Lambda_n\to \Lambda_{n-1}$ sending $p(z_1,\ldots,z_n)$ to $p(z_1,\ldots,z_{n-1},0)$.) The $P_\lambda(\theta)$ are called Jack's symmetric functions and form a linear basis in $\Lambda$.

As indicated in the introduction, the definition of the super-Jack polynomials relies on the fact that $\Lambda$ is freely generated by the power sums
\begin{equation} 
\label{FormalPowerSums}
p_r = \sum_j z_j^r,\ \ \ r\geq 1.
\end{equation} 
Defining the automorphism $\omega_\theta:\Lambda\to\Lambda$ by
$$
\omega_\theta(p_r) = (-1)^{r-1}\theta p_r,\ \ \ r\geq 1,
$$
we recall the duality relation
\begin{equation}\label{PQ}
\omega_{1/\theta}\big(P_\lambda(\theta)\big) = Q_{\lambda^\prime}(1/\theta),
\end{equation}
where
\begin{equation}\label{Qb}
Q_\lambda(\theta) = b_\lambda(\theta)P_\lambda(\theta),\ \ \ b_\lambda(\theta) = \prod_{(j,k)\in\lambda}\frac{\lambda_j-k+\theta(\lambda^\prime_k-j+1)}{\lambda_j-k+1+\theta(\lambda^\prime_k-j)}.
\end{equation}
For later reference, we note that
\begin{equation}\label{bDual}
b_{\lambda^\prime}(1/\theta) = 1/b_\lambda(\theta),
\end{equation}
which entails $\omega_{1/\theta}=\omega_\theta^{-1}$.

Considering the product expansion
\begin{equation}\label{Pprod}
P_\mu(\theta) P_\nu(\theta) = \sum_\lambda f^\lambda_{\mu\nu}(\theta)P_\lambda(\theta),
\end{equation}
we recall from \cite{Sta89} and Sections VI.7 and VI.10 in \cite{Mac95} that the coefficients $f^\lambda_{\mu\nu}$ have, in particular, the following properties:
\begin{enumerate}
\item $f^{\mu+\nu}_{\mu\nu}=1$,
\item $f^\lambda_{\mu\nu}(\theta) = f^{\lambda^\prime}_{\mu^\prime\nu^\prime}(1/\theta) b_\lambda(\theta)/b_\mu(\theta)b_\nu(\theta)$,
\item $f^\lambda_{\mu\nu}=0$ unless $|\lambda| = |\mu|+|\nu|$, $\mu,\nu\subseteq\lambda$ and $\mu\cup\nu\leq\lambda\leq\mu+\nu$.
\end{enumerate}
Here $(\mu+\nu)_j=\mu_j+\nu_j$ and $\mu\cup\nu$ is obtained by arranging the parts of $\mu$ and $\nu$ in descending order. We also recall that the skew Jack functions $P_{\lambda/\mu}$ can be defined by
\begin{equation}\label{skewP}
P_{\lambda/\mu}(\theta) = \sum_\nu f^{\lambda^\prime}_{\mu^\prime\nu^\prime}(1/\theta)P_\nu(\theta).
\end{equation}
Note that $P_{\lambda/\mu}$ is non-zero only if $\mu\subseteq\lambda$, in which case it is homogeneous of degree $|\lambda|-|\mu|$. Moreover, in the case of finitely many variables an explicit combinatorial formula entails
\begin{equation}\label{sPvan}
P_{\lambda/\mu}(z_1,\ldots,z_n)\equiv 0\ \ \mathrm{unless}\ 0\leq \lambda^\prime_j-\mu^\prime_j\leq n\ \mathrm{for~all}\ j\geq 1.
\end{equation}

\subsection{The super-Jack polynomials}\label{sec2.2} 
For $n,m\in\N$, Sergeev and Veselov \cite{SV04} (see Theorem 2) proved that the algebra $\Lambda_{n,m,\theta}$ (c.f.~\eqref{Lam}--\eqref{qinv}) is generated by the deformed power sums
$$
p_{r}(z,w;\theta) := \sum_{j=1}^n z_j^r - \frac{1}{\theta}\sum_{k=1}^m w_k^r,\ \ \ r\geq 1,
$$
and therefore $\varphi_{n,m}:p_r\mapsto p_{r}(z,w;\theta)$ defines a surjective homomorphism $\varphi_{n,m}:\Lambda\to\Lambda_{n,m,\theta}$. Letting $H_{n,m}$ denote the set of partitions with diagram contained in the fat $(n,m)$-hook, i.e.
\begin{equation} 
\label{Hnm}
H_{n,m} = \{\lambda\in\mathscr{P} : \lambda_{n+1}\leq m\},
\end{equation} 
they proved in Theorem 2 of \cite{SV05} that
$$
\ker\varphi_{n,m} = \spn\{P_\lambda : \lambda\notin H_{n,m}\}.
$$
Defining the super-Jack polynomials $SP_\lambda(z,w;\theta)$ in $n+m$ variables $z=(z_1,\ldots,z_n)$ and $w=(w_1,\ldots,w_m)$ by
\begin{equation}\label{SPDef}
SP_\lambda(z,w;\theta) := \varphi_{n,m}\big(P_\lambda(\theta)\big),
\end{equation}
it followed that the $SP_\lambda(z,w;\theta)$ with $\lambda\in H_{n,m}$ form a linear basis in $\Lambda_{n,m,\theta}$, whereas $SP_\lambda(z,w;\theta)\equiv 0$ whenever $\lambda\notin H_{n,m}$. We note that the super-Jack polynomials were first introduced by Kerov, Okounkov and Olshanski \cite{KOO98} in the setting of infinitely many variables $z=(z_1,z_2,\ldots)$ and $w=(w_1,w_2,\ldots)$, and that the terminology super-Jack polynomials appeared shortly thereafter in \cite{Oko97}. Moreover, when $m=0$,  $SP_\lambda(z,w;\theta)$ reduces to the ordinary Jack polynomial $P_\lambda(z;\theta)$, whereas $n=0$ yields the dual Jack polynomial $Q_{\lambda^\prime}(w;1/\theta)$. 

Given a partition $\lambda\in H_{n,m}$, we let
\begin{equation}
\label{es}
\begin{split} 
e(\lambda;n,m) &= (\langle\lambda_1-m\rangle,\ldots,\langle\lambda_n-m\rangle),\\ 
s(\lambda;n,m) &= (\langle\lambda^\prime_1-n\rangle,\ldots,\langle\lambda^\prime_m-n\rangle),
\end{split} 
\end{equation}
where
$$
\langle a\rangle = \max(0,a),
$$
so that  $e(\lambda;n,m)$ and $s(\lambda;n,m)$ correspond to the boxes in the diagram of $\lambda$ located to the east resp.~south of the $m\times n$ rectangle $(m^n)$; see Figure \ref{fig_fat_hook_partition}. 
We note that a partition $\lambda$ of the form \eqref{lambdamunu} is the unique partition with $e(\lambda;n,m)=\mu$ and $s(\lambda;n,m)=\nu$.
For later reference, we note
\begin{equation}\label{esDual}
\begin{split}
e(\lambda^\prime;n,m) = s(\lambda;m,n),\quad s(\lambda^\prime;n,m) = e(\lambda;m,n). 
\end{split}
\end{equation}
We often omit $n,m$ and write $e(\lambda)$ and $s(\lambda)$ if there is no danger of confusion. 

\begin{figure}[H]
\includegraphics[width=0.95\columnwidth]{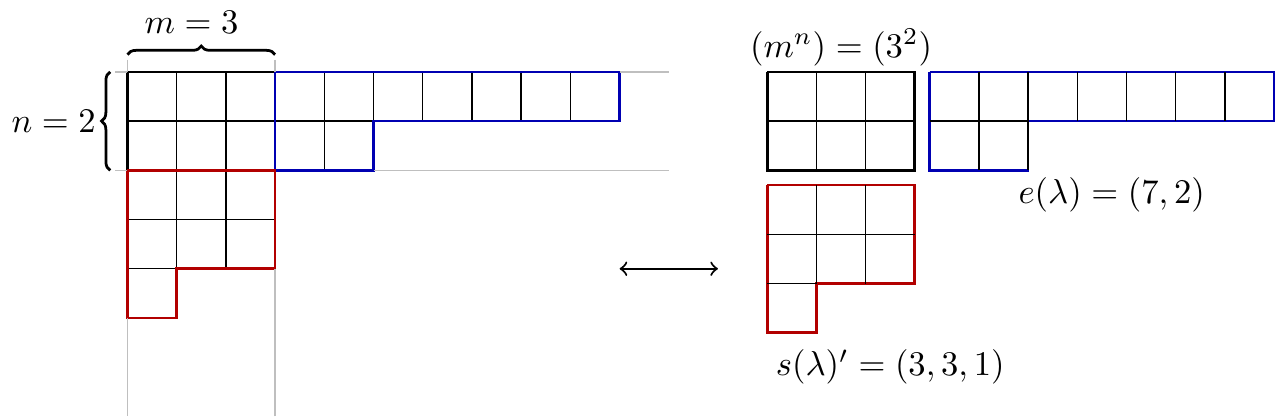}
\caption{The partition $\lambda=(10,5,3,3,1)$, which satisfies the condition \eqref{box} for $(n,m)=(2,3)$, along with its east and south components $e(\lambda)$ resp.~$s(\lambda)$, c.f.~\eqref{es}.}%
\label{fig_fat_hook_partition}
\end{figure}

The next result, which provides an expansion of super-Jack polynomials in terms of skew- and dual Jack polynomials (c.f.~\eqref{Qb} resp.~\eqref{skewP}), is essentially contained in \cite{SV05}, but, since it is a key ingredient in Section \ref{sec3}, we include a short proof.

\begin{lemma}\label{Lemma:SPExp}
For $\lambda\in H_{n,m}$, we have
\begin{equation}\label{SPExp}
SP_\lambda(z,w;\theta) = \sum_\mu(-1)^{|\mu|}P_{\lambda/\mu^\prime}(z;\theta)Q_\mu(w;1/\theta),
\end{equation}
with the sum running over all partitions $\mu$ satisfying
$$
s(\lambda)\subseteq\mu\subseteq s((\lambda^\prime+(n^m))^\prime),
$$
where $s((\lambda^\prime+(n^m))^\prime)=(\lambda^\prime_1,\ldots,\lambda^\prime_m)$.
\end{lemma}

\begin{proof}
Letting $z=(z_1,z_2,\ldots)$ and $w=(w_1,w_2,\ldots)$ be two infinite sets of variables, we start from the expansion
$$
P_\lambda(z,w;\theta) = \sum_{\mu\subseteq\lambda^\prime}P_{\lambda/\mu^\prime}(z;\theta)P_{\mu^\prime}(w;\theta),
$$
see e.g.~Section VI.7 in \cite{Mac95}. Applying, with respect to the variables $w$, the automorphism $\sigma_\theta:\Lambda\to\Lambda$ given by
$$
\big(\sigma_\theta(p_r)\big)(w) = \big(\omega_{1/\theta}(p_r)\big)(-w) = -\frac{1}{\theta}p_r(w),\ \ \ r\geq 1,
$$
and setting $z_j=w_k=0$ for $j>n$ and $k>m$, we get
$$
SP_\lambda(z,w;\theta) = \sum_{\mu\subseteq\lambda^\prime}(-1)^{|\mu|}P_{\lambda/\mu^\prime}(z;\theta)Q_{\mu}(w;1/\theta),
$$
c.f.~\eqref{PQ}. We observe that $\mu\not\subseteq s((\lambda^\prime+(n^m))^\prime)$ entails $\mu_{m+1}\neq 0$, and consequently that $Q_\mu(w_1,\ldots,w_m)\equiv 0$. On the other hand, $s(\lambda)\not\subseteq\mu$ if and only if $\lambda^\prime_j-\mu_j>n$ for some $j=1,\ldots,m$, so that $P_{\lambda/\mu^\prime}(z_1,\ldots,z_n)\equiv 0$, c.f.~\eqref{sPvan}.
\end{proof}

We note the duality relation
\begin{equation}\label{dualRel}
SP_\lambda(z,w;\theta) = (-1)^{|\lambda|}SQ_{\lambda^\prime}(w,z;1/\theta)
\end{equation}
readily inferred from \eqref{PQ}, \eqref{SPDef} and the identity $-\theta p_{r}(z,w;\theta)=p_{r}(w,z;1/\theta)$; see e.g.~Proposition~6.1 in \cite{DH12}.

\subsection{Deformed Calogero-Moser-Sutherland operators}
As previously mentioned, the functions $\Psi_\lambda$ \eqref{Psi} are eigenfunctions not only of the deformed CMS operator $\cH_{n,m,\theta}$ \eqref{dCMSOp} but of all quantum integrals constructed by Sergeev and Veselov \cite{SV04}. To describe their construction, which will be used in Section \ref{sec3}, it is convenient to set
\begin{equation}
\label{z}
z_{n+j} = w_j,\ \ \ j=1,\ldots,m
\end{equation}
and distinguish between the first $n$ and the last $m$ variables by the value of the parity function
\begin{equation}
\label{p}
p(j) := \left\{
\begin{array}{ll}
0, & j=1,\ldots,n\\
1, & j=n+1,\ldots,n+m 
\end{array}\right. .
\end{equation}
Letting
$$
\partial^{(1)}_j=(-\theta)^{p(j)}z_j\frac{\partial}{\partial z_j},
$$
PDOs $\partial^{(r)}_j$ of order $r>1$ are defined recursively by
\begin{equation}\label{par}
\partial^{(r)}_j=\partial^{(1)}_j\partial^{(r-1)}_j-\frac{1}{2}\sum_{k\neq j}(-\theta)^{1-p(k)}\frac{z_j+z_k}{z_j-z_k}\big(\partial^{(r-1)}_j-\partial^{(r-1)}_k\big).
\end{equation}
With $e_j$, $j=1,\ldots,n+m$, the standard basis vectors in $\R^{n+m}$, we consider the vector
$$
\rho_\theta = \frac{1}{2}\sum_{1\leq j<k\leq n+m}(-\theta)^{1-p(j)-p(k)}(e_j-e_k)
$$
and the $\theta$-dependent indefinite inner product given by
$$
(u,v)_\theta = \sum_{j=1}^nu_jv_j - \theta\sum_{k=1}^mu_{n+k}v_{n+k}.
$$
From Lemma 1.1 in \cite{Ser02}, we have
\begin{equation*}
\begin{split}
\cL_{n,m,\theta} &:= -\Psi_0^{-1}\circ\big(\cH_{n,m,\theta} - (\rho_\theta,\rho_\theta)_\theta\big)\circ\Psi_0\\
&= \sum_{j=1}^{n+m}(-\theta)^{p(j)}\left(z_j\frac{\partial}{\partial z_j}\right)^2\\
&\quad - \sum_{1\leq j<k\leq n+m}(-\theta)^{1-p(j)-p(k)}\frac{z_j+z_k}{z_j-z_k}\left((-\theta)^{p(j)}z_j\frac{\partial}{\partial z_j}-(-\theta)^{p(k)}z_k\frac{\partial}{\partial z_k}\right).
\end{split}
\end{equation*}
Comparing the latter expression for $\cL_{n,m,\theta}$ with \eqref{par}, it is readily verified that
$$
\cL_{n,m,\theta} = \sum_{j=1}^{n+m}(-\theta)^{p(j)}\partial^{(2)}_j.
$$
By direct and rather lengthy computations, Sergeev and Veselov \cite{SV04} proved that all of the PDOs
\begin{equation}\label{cLr}
\cL^{(r)}_{n,m,\theta} := \sum_{j=1}^{n+m}(-\theta)^{-p(j)}\partial^{(r)}_j,\ \ \ r\in\N,
\end{equation}
pairwise commute:
$$
\left[\cL^{(r)}_{n,m,\theta},\cL^{(s)}_{n,m,\theta}\right]=0,\ \ \ \forall r,s\in\N.
$$
An elegant proof of this result, relying on a notion of a Dunkl operator `at infinity', was recently obtained by the same authors in \cite{SV15}.

Moreover, the super-Jack polynomials $SP_\lambda(z,w;\theta)$, $\lambda\in H_{n,m}$, satisfy PDEs
\begin{equation}\label{PDEs}
\cL^{(r)}_{n,m,\theta}SP_\lambda=\cZ_r(\nu_\lambda-\rho_\theta)SP_\lambda
\end{equation}
with
$$
\nu_\lambda = (e(\lambda+(m^n)),s(\lambda))
$$
and $\cZ_r$ polynomials of the form
$$
\cZ_r(u,v) = \sum_{j=1}^nu_j^r + (-\theta)^r\sum_{k=1}^mv_k^r + \mathrm{l.d.t.},
$$
where $\mathrm{l.d.t.}$ is short for lower degree terms. Crucially, for our purposes, the joint eigenvalues $(\cZ_r(\nu_\lambda-\rho_\theta))_{r\in\N}$ separate the partitions in $H_{n,m}$, i.e.
\begin{equation}\label{cZEq}
\cZ_r(\nu_\lambda-\rho_\theta) = \cZ_r(\nu_\mu-\rho_\theta)\, \forall r\geq 1 \Leftrightarrow \lambda = \mu.
\end{equation}
These last two results are established in Section 6 of \cite{SV05}.

\section{Orthogonality and norms}\label{sec3}
This section contains our main results: orthogonality relations for the super-Jack polynomials and explicit formulae for their (quadratic) norms with respect to a natural generalisation of Macdonald's inner product \eqref{MacProd}.

We let $\mathbb{C}(z,w)$ denote the field of complex rational functions in $n+m$ variables $z=(z_1,\ldots,z_n)$ and $w=(w_1,\ldots,w_m)$. Given any $f\in\C(z,w)$, we introduce its conjugate $f^*\in\C(z,w)$ by
\begin{equation}\label{conj}
f^*(z,w)=\overline{f(\bar{z}^{-1},\overline{w}^{-1})},
\end{equation}
where
$$
\bar{z}^{-1} := (1/\bar{z}_1,\ldots,1/\bar{z}_n),\ \ \ \overline{w}^{-1} := (1/\overline{w}_1,\ldots,1/\overline{w}_m).
$$
Recalling the tori $T_\xi^n$ \eqref{tori}, the weight function $\Delta_{n,m}$ \eqref{Delnm} and the algebra $\Lambda_{n,m,\theta}$ \eqref{Lam}--\eqref{qinv}, we proceed to define the relevant sesquilinear product. For technical reasons, we first consider $z\in T_\xi^n$ and $w\in T_{\xi^\prime}^m$ with $\xi^\prime>\xi$. Once our main results have been proved, we will be able to relax this condition and allow also $\xi^\prime<\xi$; see Corollary \ref{Cor:indep}. The definition now follows.

\begin{definition}\label{Def:prod}
Letting $\theta>0$ and $\xi^\prime>\xi>0$, we define the sesquilinear product
$$
\langle\cdot,\cdot\rangle^\prime_{n,m,\theta}:\Lambda_{n,m,\theta}\times\Lambda_{n,m,\theta}\to\C
$$
by setting
\begin{equation}\label{HermProd}
\langle p,q\rangle^\prime_{n,m,\theta} = \frac{1}{n!m!}\int_{T_\xi^n}d\measure_n(z) \int_{T_{\xi^\prime}^m}d\measure_m(w)\, \De_{n,m}(z,w;\theta)p(z,w)q^*(z,w)
\end{equation}
for all $p,q\in\Lambda_{n,m,\theta}$.
\end{definition}

Due to the following result, the specific choice of $\xi,\xi^\prime$ has no effect on the values of the product $\langle\cdot,\cdot\rangle^\prime_{n,m,\theta}$.

\begin{lemma}\label{Lemma:indep}
For $\theta>0$ and any $p,q\in\Lambda_{n,m,\theta}$, the value of $\langle p,q\rangle^\prime_{n,m,\theta}$ is independent of $\xi,\xi^\prime>0$ provided that $\xi^\prime>\xi$.
\end{lemma}
\begin{proof}
For $k\in\Z_{\geq 0}$, we let $P_{n,m}^k\subset \C[z,w]$ denote the subspace of homogeneous (complex) polynomials of degree $k$ in $n+m$ independent variables. Extending $\langle\cdot,\cdot\rangle^\prime_{n,m,\theta}$ in the obvious way to all of $\C[z,w]$, we claim that
$$
\big\langle P_{n,m}^k,P_{n,m}^l\big\rangle^\prime_{n,m,\theta} = 0\ \ \mathrm{if}\ \ k\neq l.
$$
Taking this claim for granted, we may restrict attention to $p,q\in P_{n,m}^k$. Substituting $z\to\xi z$ and $w\to\xi w$ in \eqref{HermProd}, using \eqref{scaleInv} (and similarly for $\Delta_m(w;1/\theta)$ and $d\measure_m(w)$) as well as homogeneity of $p,q$, we see that no generality is lost by setting $\xi=1$, say. Now, let us assume that $1/\theta\in\N$, so that $\Delta_m(w;1/\theta)$ extends to an analytic function on $(\C\setminus\{0\})^m$. Independence of $\xi^\prime>1$ is then a simple consequence of Cauchy's theorem. Finally, by appealing to Carlson's theorem (see e.g.~Section~5.8 in \cite{Tit39}), we can extend this result to all $\theta>0$. A sufficient bound on $|\langle p,q\rangle^\prime_{n,m,\theta}|$ (with respect to $\theta$) is readily obtained by combining the observation
$$
|\langle p,q\rangle^\prime_{n,m,\theta}| < C\langle 1,1\rangle^\prime_{n,\theta}\cdot \langle 1,1\rangle^\prime_{m,1/\theta},
$$
where $C\equiv C(\xi,\xi^\prime,p,q)>0$, with the explicit norms formulae \eqref{Nlam} for $\lambda=\emptyset$ and Stirling's formula for the Gamma function.

At this point, there remains only to verify our claim. 
Considering the Euler operator 
$$
E = \sum_{j=1}^n z_j\frac{\partial}{\partial z_j} + \sum_{k=1}^m w_k\frac{\partial}{\partial w_k}
$$
(which coincides with $\cL^{(1)}_{n,m,\theta}$), a direct computation reveals that
$$
E\Delta_{n,m}(z,w;\theta) = 0.
$$
In addition, we have
$$
Ef^*=-(Ef)^*,\ \ \ f\in\C[z,w].
$$
Let $p\in P_{n,m}^k$ and $q\in P_{n,m}^l$. From the above, we infer
\begin{multline*}
E\Delta_{n,m}(z,w;\theta)p(z,w)q^*(z,w)\\
= \Delta_{n,m;\theta}(z,w)\big[q^*(z,w)Ep(z,w) - p(z,w)(Eq)^*(z,w)\big],
\end{multline*}
and therefore
$$
\langle Ep,q\rangle^\prime_{n,m,\theta} = \langle p,Eq\rangle^\prime_{n,m,\theta}.
$$
Since $Ep=kp$ and $Eq=lq$, the claim follows.
\end{proof}

\begin{remark}
An alternative proof of the Lemma can be obtained as follows: first, rewrite the weight function $\Delta_{n,m}(z,w;\theta)$ as in the proof of Theorem \ref{Thm:orth}; second, substitute, e.g., the expansion
$$
\prod_{j=1}^n\prod_{k=1}^m(1-z_j/w_k)^{-1} = \sum_\lambda s_\lambda(z)s_\lambda(w^{-1}),
$$
where $s_\lambda$ are the Schur polynomials; third, write $p$ and $q$ in the form
$$
p(z,w) = \sum_l p_{1,l}(z)p_{2,l}(w),\ \ \ q(z,w) = \sum_l q_{1,l}(z)q_{2,l}(w);
$$
and, finally, use the fact that the value of Macdonald's inner product (when written as in \eqref{MacProdAlt}) is independent of $\xi>0$.
\end{remark}

At a later stage, we shall see that we need, in fact, only require $\xi^\prime\neq\xi$. Although the result of Lemma \ref{Lemma:indep} remains valid for any polynomials $p,q\in\C[z,w]$, it is readily seen (already in the $n=m=1$ case) that this stronger claim does not.

Next, we prove that that the deformed CMS operators $\cL^{(r)}_{n,m,\theta}$ \eqref{cLr} are symmetric with respect to $\langle\cdot,\cdot\rangle^\prime_{n,m,\theta}$.

\begin{proposition}
Let $\theta>0$. Then we have
\begin{equation}\label{sym}
\left\langle\cL^{(r)}_{n,m,\theta}\, p,q\right\rangle^\prime_{n,m,\theta}=\left\langle p,\cL^{(r)}_{n,m,\theta}\, q\right\rangle^\prime_{n,m,\theta}
\end{equation}
for all $r\in\N$ and $p,q\in\Lambda_{n,m,\theta}$.
\end{proposition}

\begin{proof}
Rewriting the weight function $\De_{n,m}$ in terms of the notation \eqref{z} and the parity function $p(j)$ \eqref{p}, we obtain the expression
\begin{equation}
\De_{n,m}(z;\theta) = \prod_{1\leq j\neq k\leq n+m}(1-z_j/z_k)^{-(-\theta)^{1-p(j)-p(k)}}.
\end{equation}
For what follows, we need to extend the definition of the product $\langle\cdot,\cdot\rangle^\prime_{n,m,\theta}$ to suitable rational functions. 
More precisely, for $f,g\in\C(z,w)$, we let \eqref{HermProd} define $\langle f,g\rangle_{n,m,\theta}^{\prime}$ whenever the right-hand side is well-defined.

To prove \eqref{sym}, it clearly suffices to establish the equality
\begin{multline}
\label{sumEq}
\sum_{j=1}^{n+m}(-\theta)^{-p(j)}\left\langle\partial^{(r-s)}_jp,\partial^{(s)}_jq\right\rangle^\prime_{n,m,\theta}\\
= \sum_{j=1}^{n+m}(-\theta)^{-p(j)}\left\langle\partial^{(r-s-1)}_jp,\partial^{(s+1)}_jq\right\rangle^\prime_{n,m,\theta}
\end{multline}
for $0\leq s\leq r-1$. From Section~5 of \cite{SV15}, we recall that
each PDO $\partial^{(r)}_j$ maps $\Lambda_{n,m,\theta}$ into $\Lambda_{n,m,\theta}[z_j]$ (the algebra of polynomials in $z_j$ with coefficients from $\Lambda_{n,m,\theta}$), which ensures that each term is well-defined. Using integration by parts in $z_j$, we deduce that
\begin{multline}
\label{parEq}
\left\langle \partial^{(1)}_j\partial^{(r-s-1)}_jp,\partial^{(s)}_jq\right\rangle^\prime_{n,m,\theta} = \left\langle \partial^{(r-s-1)}_jp,\partial^{(1)}_j\partial^{(s)}_jq\right\rangle^\prime_{n,m,\theta}\\
-\left\langle \big(\De_{n,m}^{-1}\partial^{(1)}_j\De_{n,m}\big)\partial_{j}^{(r-s-1)}p,\partial^{(s)}_jq\right\rangle^\prime_{n,m,\theta}.
\end{multline}
By a direct computation, we find that
$$
\De_{n,m}^{-1}\partial^{(1)}_j\De_{n,m}=-\sum_{k\neq j}(-\theta)^{1-p(k)}\frac{z_j+z_k}{z_j-z_k}, \quad j = 1,\ldots,n+m.
$$
Hence $\theta>0$ guarantees that $\partial^{(1)}_j\De_{n,m}$ is (locally) integrable, so that the latter term in the right-hand side of \eqref{parEq} is well-defined. Recalling \eqref{par}, we thus infer
\begin{multline*}
\left\langle\partial^{(r-s)}_jp,\partial^{(s)}_jq\right\rangle^\prime_{n,m,\theta} = \left\langle\partial^{(r-s-1)}_jp,\partial^{(1)}_j\partial^{(s)}_jq\right\rangle^\prime_{n,m,\theta}\\
-\frac{1}{2}\sum_{k\neq j}(-\theta)^{1-p(k)}\left\langle\big(\partial^{(r-s-1)}_j+\partial^{(r-s-1)}_k\big)p,\frac{z_j+z_k}{z_j-z_k}\partial^{(s)}_jq\right\rangle^\prime_{n,m,\theta}.
\end{multline*}
Multiplying by $(-\theta)^{-p(j)}$ and taking the sum over $j$, we use invariance under the interchange $j\leftrightarrow k$ to obtain \eqref{sumEq}.
\end{proof}

Recalling the definitions of $N_n(\lambda;\theta)$ \eqref{Nlam}, $b_\lambda(\theta)$ \eqref{Qb} and $H_{n,m}$ \eqref{Hnm}, we are now ready to state and prove our main result.

\begin{theorem}\label{Thm:orth}
For $\theta>0$, the super-Jack polynomials $SP_\lambda(z,w;\theta)$, $\lambda\in H_{n,m}$,  satisfy the orthogonality relations
\begin{equation}\label{orth}
\langle SP_\lambda,SP_\mu\rangle^\prime_{n,m,\theta} = \delta_{\lambda\mu}N_{n,m}(\lambda;\theta)
\end{equation}
with (quadratic) norms
\begin{equation}\label{qnorm1}
N_{n,m}(\lambda;\theta) = 0\ \ \mathrm{if}\ \  (m^n)\not\subseteq\lambda
\end{equation}
and
\begin{equation}\label{qnorm2}
N_{n,m}(\lambda;\theta) = N_n(e(\lambda);\theta)N_m(s(\lambda);1/\theta) \frac{b_{e(\lambda)}(\theta)b_{s(\lambda)}(1/\theta)}{b_\lambda(\theta)}\ \ \mathrm{if}\ 
 (m^n)\subseteq\lambda.
\end{equation}
\end{theorem}

\begin{proof}
Specialising \eqref{sym} to $p=SP_\lambda$ and $q=SP_\mu$, we get from \eqref{PDEs} that
$$
\big[\cZ_r(\nu_\lambda-\rho_\theta) - \cZ_r(\nu_\mu-\rho_\theta)\big]\langle SP_\lambda,SP_\mu\rangle^\prime_{n,m,\theta} = 0.
$$
The orthogonality relations \eqref{orth} are now a direct consequence of \eqref{cZEq}

Exploiting Lemma \ref{Lemma:indep}, we shall compute the norms
\begin{multline}\label{inProdExpr}
N_{n,m}(\lambda;\theta)\\
= \frac{1}{n!m!}\int_{T_\xi^n}d\measure_n(z) \int_{T_{\xi^\prime}^m}d\measure_m(w)\, \De_{n,m}(z,w;\theta)SP_\lambda(z,w;\theta)SP_\lambda(z^{-1},w^{-1};\theta)
\end{multline}
by taking $\xi^\prime\to+\infty$ while keeping $\xi$ fixed. Rewriting $\Delta_{n,m}$ \eqref{Delnm} according to
$$
\Delta_{n,m}(z,w;\theta) = (-1)^{nm}\frac{(z_1\cdots z_n)^m}{(w_1\cdots w_m)^n}\Delta_n(z;\theta)\Delta_m(w;1/\theta)\prod_{j=1}^n\prod_{k=1}^m (1-z_j/w_k)^{-2},
$$
it is clear from the elementary estimate
$$
(1-t)^{-2} = 1+O(t),\ \ \ t\to 0,
$$
that, for $(z,w)\in T^n_\xi\times T^m_{\xi^\prime}$, 
\begin{multline}\label{DelBd}
\Delta_{n,m}(z,w;\theta)\\
= (-1)^{nm}\frac{(z_1\cdots z_n)^m}{(w_1\cdots w_m)^n}\Delta_n(z;\theta)\Delta_m(w;1/\theta)\big(1+O((1/\xi^\prime)^{nm})\big),\ \ \ \xi^\prime\to+\infty. 
\end{multline}
Furthermore, since $Q_\mu$ is homogenous of degree $|\mu|$, Lemma~\ref{Lemma:SPExp} entails
\begin{multline}\label{SPprodBd}
SP_\lambda(z,w;\theta)SP_\lambda(z^{-1},w^{-1};\theta)\\
= (-1)^{|s((\lambda^\prime+(n^m))^\prime)|+|s(\lambda)|}P_{\lambda/(s((\lambda^\prime+(n^m))^\prime)^\prime}(z;\theta)P_{\lambda/s(\lambda)^\prime}(z^{-1};\theta)\\
\cdot Q_{s((\lambda^\prime+(n^m))^\prime)}(w;1/\theta)Q_{s(\lambda)}(w^{-1};1/\theta) + O\big((\xi^\prime)^{|s((\lambda^\prime+(n^m))^\prime)|-|s(\lambda)|-1}\big),\ \ \ \xi^\prime\to+\infty
\end{multline}
when $(z,w)\in T^n_\xi\times T^m_{\xi^\prime}$. At this point, we recall the orthogonality relations \eqref{Porth} for the Jack polynomials and the scale invariance \eqref{scaleInv} of $\Delta_n$. Combining the bounds \eqref{DelBd}--\eqref{SPprodBd} with the homogeneity of $P_{\lambda/\mu^\prime}$ and $Q_\mu$, we thus infer 
\begin{multline*}
N_{n,m}(\lambda;\theta) = (-1)^{|s((\lambda^\prime+(n^m))^\prime)|+|s(\lambda)|+nm}\left(\frac{\xi}{\xi^\prime}\right)^{nm-|s((\lambda^\prime+(n^m))^\prime)|+|s(\lambda)|}\\
\cdot \big\langle(z_1\cdots z_n)^mP_{\lambda/s((\lambda^\prime+(n^m))^\prime)^\prime},P_{\lambda/s(\lambda)^\prime}\big\rangle^\prime_{n,\theta}\cdot \big\langle Q_{s((\lambda^\prime+(n^m))^\prime)},(w_1\cdots w_m)^nQ_{s(\lambda)}\big\rangle^\prime_{m,1/\theta}\\
+ O\big((\xi^\prime)^{|s((\lambda^\prime+(n^m))^\prime)|-|s(\lambda)|-nm-1}\big),\ \ \ \xi^\prime\to+\infty.
\end{multline*}
We observe
$$
|s((\lambda^\prime+(n^m))^\prime)|-|s(\lambda)| \left\{\begin{array}{ll} = nm, & \mathrm{if}~(m^n)\subseteq\lambda \\ <nm, & \mathrm{otherwise}\end{array}\right.
$$
Taking the limit $\xi^\prime\to+\infty$, it follows that $N_{n,m}(\lambda;\theta)=0$ whenever $(m^n)\not\subseteq\lambda$.
To compute the remaining norms, we note
$$
(w_1\cdots w_m)^nQ_{s(\lambda)}(w;1/\theta) = b_{s(\lambda)}(1/\theta)P_{s(\lambda)+(n^m)}(w;1/\theta)
$$
and
$$
N_m(s(\lambda)+(n^m);1/\theta)=N_m(s(\lambda);1/\theta).
$$
For $(m^n)\subseteq\lambda$, 
it follows that
\begin{equation}\label{SPprodExp}
\begin{split}
N_{n,m}(\lambda;\theta) &= b_{s(\lambda)}(1/\theta)b_{s(\lambda)+(m^n)}(1/\theta)N_m(s(\lambda);1/\theta)\\
&\quad \cdot \big\langle(z_1\cdots z_n)^mP_{\lambda/s((\lambda^\prime+(n^m))^\prime)^\prime},P_{\lambda/s(\lambda)^\prime}\big\rangle^{\prime}_{n,\theta}.
\end{split}
\end{equation}
The definition \eqref{skewP} entails
\begin{equation}\label{skewPexp1}
(z_1\cdots z_n)^m P_{\lambda/s((\lambda^\prime+(n^m))^\prime)^\prime}(z;\theta) = \sum_\nu f^{\lambda^\prime}_{s((\lambda^\prime+(n^m))^\prime)^\prime\nu^\prime}(1/\theta)P_{\nu+(m^n)}(z;\theta).
\end{equation}
We note that the maximal term (with respect to the dominance order) in the right-hand side is given by
$$
\nu = e(\lambda)\Leftrightarrow s((\lambda^\prime+(n^m))^\prime)\cup\nu^\prime = \lambda^\prime.
$$
Indeed, $f^{\lambda^\prime}_{\lambda^*\nu^\prime}=0$ unless $\lambda^*\cup\nu^\prime\leq\lambda^\prime$ and
\begin{equation*}
\begin{split}
f^{\lambda^\prime}_{s((\lambda^\prime+(n^m))^\prime),e(\lambda)^\prime}(1/\theta) &= f^\lambda_{\lambda-e(\lambda),e(\lambda)}(\theta) \frac{b_{\lambda-e(\lambda)}(\theta)b_{e(\lambda)}(\theta)}{b_\lambda(\theta)}\\
&= \frac{b_{\lambda-e(\lambda)}(\theta)b_{e(\lambda)}(\theta)}{b_\lambda(\theta)}\neq 0.
\end{split}
\end{equation*}
On the other hand, since $f^{\lambda^\prime}_{s(\lambda)\nu^\prime}=0$ unless $s(\lambda)+\nu^\prime\geq\lambda^\prime$ and
$$
f^{\lambda^\prime}_{s(\lambda),(e(\lambda)+(n^m))^\prime} = f^{\lambda^\prime}_{s(\lambda),\lambda^\prime-s(\lambda)} = 1,
$$
the minimal term in the right-hand side of
\begin{equation}\label{skewPexp2}
P_{\lambda/s(\lambda)^\prime}(z^{-1};\theta) = \sum_\nu f^{\lambda^\prime}_{s(\lambda),\nu^\prime}(1/\theta)P_{\nu}(z^{-1};\theta)
\end{equation}
corresponds to
$$
\nu = e(\lambda)+(n^m)\Leftrightarrow s(\lambda)+\nu^\prime = \lambda^\prime.
$$
The upshot is that when we substitute the expansions \eqref{skewPexp1} and \eqref{skewPexp2} into the right-hand side of \eqref{SPprodExp}, the orthogonality relations \eqref{Porth} for the Jack polynomials $P_\lambda$ imply that a single term yields a non-zero contribution. More precisely, we get
\begin{equation*}
\begin{split}
N_{n,m}(\lambda;\theta) &= b_{s(\lambda)}(1/\theta)b_{s(\lambda)+(m^n)}(1/\theta)N_{m}(s(\lambda);1/\theta)\\
&\quad \cdot \frac{b_{\lambda-e(\lambda)}(\theta)b_{e(\lambda)}(\theta)}{b_\lambda(\theta)}N_n(e(\lambda);\theta).
\end{split}
\end{equation*}
Using
$$
b_{s(\lambda)+(m^n)}(1/\theta) = b_{(\lambda-e(\lambda))^\prime}(1/\theta) = 1/b_{\lambda-e(\lambda)}(\theta),
$$
we arrive at \eqref{qnorm2}. This concludes the proof of the theorem.
\end{proof}

We note that all factors in the right-hand side of \eqref{qnorm2} are manifestly positive (c.f.~\eqref{Qb} and \eqref{Nlam}), and therefore
\begin{equation}\label{pos}
N_{n,m}(\lambda;\theta) > 0\ \ \mathrm{if}\ (m^n)\subseteq\lambda.
\end{equation}
Furthermore, \eqref{es} and \eqref{bDual} entail the duality property
$$
N_{n,m}(\lambda;\theta) = b_{\lambda^\prime}(1/\theta)^2 N_{m,n}(\lambda^\prime;1/\theta),
$$
which also can be directly inferred from \eqref{dualRel} and
$$
\Delta_{n,m}(z,w;\theta) = \Delta_{m,n}(w,z;1/\theta).
$$

Since the super-Jack polynomials $SP_\lambda(z,w;\theta)$ ($\lambda\in H_{n,m}$) span $\Lambda_{n,m,\theta}$ and their (quadratic) norms $N_{n,m}(\lambda;\theta)$ are real, the next result is immediate from Theorem \ref{Thm:orth}.

\begin{corollary}\label{Cor:Hermitian}
Letting $\theta>0$ and $p,q\in\Lambda_{n,m,\theta}$, we have
$$
\langle p,q\rangle^\prime_{n,m,\theta} = \overline{\langle q,p\rangle^\prime_{n,m,\theta}}.
$$
\end{corollary}

In other words, the sesquilinear product $\langle\cdot,\cdot\rangle^\prime_{n,m,\theta}$ is Hermitian.

As a simple consequence, we can infer the following strengthened version of Lemma \ref{Lemma:indep}.

\begin{corollary}\label{Cor:indep}
Given $\theta>0$ and any $p,q\in\Lambda_{n,m,\theta}$, the value of $\langle p,q\rangle^\prime_{n,m,\theta}$, as defined by \eqref{HermProd}, is independent of $\xi,\xi^\prime>0$ as long as $\xi^\prime\neq \xi$.
\end{corollary}

\begin{proof}
We shall use a superscript $(\xi,\xi^\prime)$ to keep track of the choice of $\xi,\xi^\prime$ in Definition \ref{Def:prod}. Taking $z\to z^{-1}$ and $w\to w^{-1}$ in \eqref{HermProd} and observing
$$
\Delta_{n,m}(z^{-1},w^{-1};\theta) = \Delta_{n,m}(z,w;\theta),
$$
$$
d\measure_n(z^{-1}) = (-1)^nd\measure_n(z),\ \ \ d\measure_m(w^{-1}) = (-1)^md\measure_m(w),
$$
we deduce
$$
\langle p,q\rangle^{\prime,(\xi,\xi^\prime)}_{n,m,\theta} = \overline{\langle q,p\rangle^{\prime,(\xi,\xi^\prime)}_{n,m,\theta}} = \langle p,q\rangle^{\prime,(1/\xi,1/\xi^\prime)}_{n,m,\theta}.
$$
Since $\xi^\prime>\xi>0$ if and only if $1/\xi>1/\xi^\prime>0$, the claim is now evident from Lemma \ref{Lemma:indep}.
\end{proof}

From Theorem \ref{Thm:orth}, we see that
$$
\ker\, \langle\cdot,\cdot\rangle^\prime_{n,m,\theta} = \spn\{SP_\lambda(z,w;\theta)  : (m^n)\not\subseteq\lambda 
\},
$$
and when we combine this with the positivity of $N_{n,m}(\lambda;\theta)$ for $(m^n)\subseteq\lambda$ 
 we arrive at the following corollary.

\begin{corollary}\label{Cor:innerProd}
The Hermitian product $\langle\cdot,\cdot\rangle^\prime_{n,m,\theta}$ descends to a positive definite inner product on the factor space
$$
V_{n,m,\theta}:= \Lambda_{n,m,\theta}\big/\spn\{SP_\lambda(z,w;\theta)  : (m^n)\not\subseteq\lambda \},
$$
and the renormalised super-Jack polynomials
$$
\widetilde{SP}_\lambda(z,w;\theta):= N_{n,m}(\lambda;\theta)^{-1/2}SP_\lambda(z,w;\theta),\ \ \ 
(m^n)\subseteq \lambda, 
$$
where $N_{n,m}(\lambda;\theta)$ is given by \eqref{qnorm2}, yield an orthonormal basis in the resulting inner product space.
\end{corollary}

\section{A Hilbert space interpretation of deformed CMS operators}\label{sec4}
We proceed to establish a Hilbert space interpretation of the deformed CMS operator \eqref{dCMSOp} as well as all other quantum integrals. 

The results of Section \ref{sec3} suggests a natural regularisation of the ill-defined $L^2$-products of the (formal) eigenfunctions $\Psi_\lambda$ \eqref{Psi} of the deformed CMS operator $\cH_{n,m,\theta}$ \eqref{dCMSOp}. More specifically, let us equip the linear space
\begin{equation*}
\mathcal{U}_{n,m,\theta}:= \big\lbrace \Psi_0(x,y;\theta)p(z(x),w(y)) : p\in\Lambda_{n,m,\theta}\big\rbrace
\end{equation*}
with the sesquilinear product given by
\begin{equation*}
(\psi,\phi)_{n,m,\theta} = \int_{([0,2\pi]+i\epsilon)^n}dx \int_{([0,2\pi]+i\epsilon^\prime)^m}dy\, \psi(x,y)\overline{\phi}(x,y),
\end{equation*}
where $\overline{\phi}(x,y) = \overline{\phi(\bar{x},\bar{y})}$ and the regularisation parameters $\epsilon,\epsilon^\prime\in\mathbb{R}$ should be chosen such that $\epsilon\neq\epsilon^\prime$. 

We note that, for $m=0$ and $\epsilon=0$, the resulting inner product above is identical to the inner product used by physicists in the application of the CMS model in quantum mechanics. Moreover, introducing polynomials $p,q\in\Lambda_{n,m,\theta}$ by requiring
\begin{equation*}
p(z(x),w(y)) = \frac{\psi(x,y)}{\Psi_0(x,y)},\ \ \ q(z(x),w(y)) = \frac{\phi(x,y)}{\Psi_0(x,y)},
\end{equation*}
it is readily verified that
\begin{equation*}
(\psi,\phi)_{n,m,\theta} = c_{n,m,\theta}\langle p,q\rangle_{n,m,\theta}^\prime,\ \ \ c_{n,m,\theta} = (2\pi)^{n+m}n!m!\Big/2^{\theta n(n-1)+m(m-1)/\theta-2nm}.
\end{equation*}
Note that $\xi=e^{-\epsilon}$ and $\xi^\prime=e^{-\epsilon^\prime}$, so that both products are well-defined as long as $\epsilon\neq\epsilon^\prime$. Hence, by Corollaries \ref{Cor:Hermitian}--\ref{Cor:innerProd}, the sesquilinear product $(\cdot,\cdot)_{n,m,\theta}$ is Hermitian, independent of the values of the regularisation parameters $\epsilon,\epsilon^\prime$, and descends to a positive definite inner product on the factor space
\begin{equation*}
\mathcal{V}_{n,m,\theta}:= \mathcal{U}_{n,m,\theta}\big/\spn\{\Psi_\lambda(x,y;\theta)  : (m^n)\not\subseteq\lambda \}.
\end{equation*}
Consequently, we can complete $\mathcal{V}_{n,m,\theta}$ to obtain a Hilbert space $\widetilde{\mathcal{V}}_{n,m,\theta}$, and the eigenfunctions $\Psi_\lambda$, with $\lambda\in H_{n,m}$ and $(m^n)\subseteq\lambda$, yield an orthogonal basis thereof. Since the corresponding eigenvalues are real-valued, we see that the deformed CMS operator $\cH_{n,m,\theta}$ induces an (essentially) self-adjoint operator in $\widetilde{\mathcal{V}}_{n,m,\theta}$. By the same reasoning, this result extends to all quantum integrals
$$
\mathcal{H}_{n,m,\theta}^{(r)}:= -\Psi_0\circ\mathcal{L}_{n,m,\theta}^{(r)}\circ\Psi_0^{-1},\ \ \ r\geq 1,
$$
c.f.~\eqref{cLr}.

\section{Conclusions and outlook}\label{sec5}
In the standard quantum mechanical interpretation of the usual ($m=0$) trigonometric CMS model, the absolute-square of a wave function has a physical interpretation as a probability amplitude. We thus want to stress that the wavefunctions \eqref{Psi} of the deformed CMS operator have, despite our Hilbert space interpretation, no such probabilistic interpretation \cite{AL17}. This is consistent with recent physics applications of the trigonometric CMS model in the context of the fractional quantum Hall effect; see e.g.\ \cite{vES98,BH08,PH14}. 

It would be interesting to generalize the results in the present paper to the $BC(n,m)$ variant of the deformed trigonometric CMS model \cite{SV04,SV09b}, and to the relativistic deformations of the CMS models in the sense of Ruijsenaars and van Diejen \cite{Rui87,vDi94} which, in the (undeformed) trigonometric case, correspond to MacDonald- and Koornwinder polynomials. Indeed, deformed variants of these models and corresponding (formal) eigenfunctions are known (see e.g.\ \cite{SV04,SV09a,SV09b,HL10,DH12}), but Hilbert space interpretations are missing. The results in the present paper provide some hints on how such interpretations could be established.

We finally note that the orthogonality relations \eqref{Porth} and (quadratic) norms formulae \eqref{Nlam} for the Jack polynomials are an important input in the derivation of integral representations for the Jack polynomials in \cite{AMOS95} (see also \cite{OO97,Saw97,KMS03} for related integral representations). We expect that, in a similar way, our orthogonality results for the super-Jack polynomials will make it possible to derive analogous integral representations for the super-Jack polynomials. We note that some tools needed for such a derivation were conjectured in \cite{AL17}, Section V. It would be interesting to work out these integral representations in detail, but this is beyond the scope of the present paper

\section*{Acknowledgments}
We would like to thank O.A.~Chalykh, M.~Noumi, A.N.~Sergeev and A.P.~Veselov for encouragement and helpful discussions. We gratefully acknowledge partial financial support by the {\em Stiftelse Olle Engkvist Byggm\"astare} (contract 184-0573). E.L.~acknowledges support by VR Grant No.~2016-05167. The work of F.A.~was partially carried out as a \emph{JSPS International Research Fellow}.

\bibliographystyle{amsalpha}

\end{document}